\documentclass[11pt,a4paper,reqno]{amsart}

\usepackage{amsfonts, amssymb, amscd}
\usepackage{amsmath}
\usepackage{verbatim}
\usepackage{amssymb}
\usepackage{mathrsfs}
\usepackage{graphicx}
\usepackage{cite}
\usepackage[all]{xy}
\usepackage{tikz}
\usepackage{tikz-cd}

\usepackage{graphicx}
\usepackage{float}
\usepackage[margin=1.5in]{geometry}
\usepackage[pagebackref]{hyperref}
\newtheorem{theorem}{Theorem}[section]
\newtheorem{lemma}[theorem]{Lemma}
\newtheorem{proposition}[theorem]{Proposition}

\newtheorem{remark}[theorem]{Remark}
\newtheorem{conjecture}[theorem]{Conjecture}

\numberwithin{equation}{section}

\DeclareMathOperator{\vol}{vol}

\def\Pic{\mathrm{Pic}}

\begin{document}

\title{Fano varieties with conjecturally largest Fano index}

\begin{abstract}
For Fano varieties of various singularities such as canonical and terminal, we construct examples with large Fano index. By low-dimensional evidence, we conjecture that our examples have the largest Fano index for all dimensions.
\end{abstract}

\author{Chengxi Wang}
\address{UCLA Mathematics Department,
Box 951555, Los Angeles, CA 90095-1555} \email{chwang@math.ucla.edu}

\maketitle

\section{introduction}

We call a normal projective variety $X$ Fano if the anti-canonical divisor $-K_X$ is ample. Del Pezzo surface is a $2-$dimensional Fano variety. A Fano variety is called $\mathbb{Q}-$Fano if it has only terminal $\mathbb{Q}-$factorial singularities and its Picard number is one. For a Fano variety, the \textit{Fano index} is defined to be:
$$\mathrm{FI}(X):=\max\{m\in \mathbb{Z}_{>0}|-K_X \thicksim_{\mathbb{Q}} mA, \text{ where } A \text{ is a } \text{Weil divisor}\}.$$
It is proved by Prokhorov that for a $\mathbb{Q}-$Fano threefold $X$, the Fano index belongs to $\{1,\ldots,11,13,17,19\}$ \cite[Theorem 1.1]{ProkhorovQ-FanoI}.
And if $\mathrm{FI}(X)=19$, then $X\simeq \mathbb{P}^3(7,5,4,3)$; if $\mathrm{FI}(X)=17$, then $X\simeq \mathbb{P}^3(7,5,3,2)$ \cite[Theorem 1.2.]{ProkhorovQ-FanoII}. 

In this paper, for various singularities such as canonical and terminal, we give 
Fano varieties with the conjecturally largest Fano index. 
Our construction uses weighted projective spaces with weights expressed by {\it Sylvester's sequence}. Sylvester's sequence is defined recursively by $s_0 = 2$ and $s_n = s_{n-1}(s_{n-1}-1) + 1$ for $n \geq 1$.
For Fano varieties with Gorenstein canoncial singularities, 
Nill gave the examples with the conjecturally largest Fano index \cite[Corollary 6.1.]{Nill2007}. We will write Nill's example as Theorem \ref{Gocanonicalindex}.

\begin{theorem}[Theorem \ref{largestindexcan}]
For each integer $n\geq 2$, there is a Fano $n-$fold with canonical singularities that has Fano index $(s_{n-1}-1)(2s_{n-1}-3)$. In particular, this is larger than $2^{2^{n-1}}$.
\end{theorem}

When $n=2$, we prove the Fano variety $\mathbb{P}(1,2,3)$ given in Theorem \ref{largestindexcan} achieves the largest Fano index $6$ among all del Pezzo surfaces with canonical singularities (Proposition \ref{cdelpezzo6}). The example constructed in Theorem \ref{largestindexcan} should have the largest possible Fano index among all canonical Fano varieties of dimension $n$ (Conjecture \ref{conjectureFindexcanonical}).

\begin{theorem}[Theorem \ref{largestindex}]
For each integer $n\geq 3$, there is a Fano $n-$fold with terminal singularities that has Fano index $\frac{1}{2}(s_{n-1}-1)^2-1$. In particular, the Fano index is larger than $2^{2^{n-1}}$.
\end{theorem}

This example should be the largest possible Fano index among such Fano varieties (Conjecture \ref{conjectureFindexterminal}).

Let $X$ be a Fano variety of dimension $n$.
The \textit{volume} of $X$ is defined to be
$$\mathrm{vol}(X) := \lim_{\ell \rightarrow \infty} h^0(X,-\ell K_X)/(\ell^n/n!)$$
which measures the asymptotic growth of the anti-plurigenera $h^0(X,-\ell K_X)$.
Moreover, the volume $\vol(X)$ equals the intersection number $(-K_X)^n.$ Varieties of general type, Calabi-Yau varieties and Fano varities with various singularities and very small volume are studied in
\cite{TW21,ETW21volume,T22}.
Balletti,  Kasprzyk, and Nill prove the weighted projective space $$\mathbb{P}^n(1,1,2(s_n-1)/s_{n-1},\ldots,2(s_n-1)/s_1)$$ has the largest volume $2(s_n-1)^2$ among all $n-$dimensional canonical toric Fano varieties for $n\geq 4$ \cite[Corollary 1.3]{BKN}. 
Kasprzyk shows that for $n\geq 2$, the weighted projective space 
$$X=\mathbb{P}^n(1,1,(s_{n-1}-1)/s_{n-2},\ldots,(s_{n-1}-1)/s_0)$$ is terminal and has very large volume $(-K_X)^n=\frac{s_{n-1}^n}{(s_{n-1}-1)^{n-2}}$ \cite[Lemma 3.7.]{Kasprzyk2013}. In particualr, it is conjectured to
have the largest possible volume among the
terminal Fano varieties of dimension $n$.

For $n-$dimensional $\mathbb{Q}-$factorial Gorensein toric Fano variety $X$ with Picard number one, Nill gives the Fano varieties with largest $(-K_X)^n$ \cite[Theorem A']{Nill2007}: if $n=2$, then $(-K_X)^n\leq 9$ with equality iff $X\cong \mathbb{P}^2$; if $n=3$, then $(-K_X)^n\leq 72$ with equiality iff $X\cong\mathbb{P}^3(3,1,1,1)$ or $X\cong \mathbb{P}^3(6,4,1,1)$; if $n\geq 4$, then $(-K_X)^n\leq 2(s_{n-1}-1)^2$ with equality iff $$X\cong \mathbb{P}(2(s_{n-1}-1)/s_0,\ldots,2(s_{n-1}-1)/s_{n-2},1,1).$$
Nill conjecuture that this theorem holds for Gorenstein canonical Fano varieties \cite[Conjecture 2.1.]{Nill2007}. Prokhorov proves the conjecture in dimension $3$ \cite[Theorem 1.5]{Prokhorov2005}. That is to say, for all Fano threefold $X$ with Gorenstein canonical singularities, the degree $(-K_X)^n\leq 72$ with equiality iff $X\cong\mathbb{P}^3(3,1,1,1)$ or $X\cong \mathbb{P}^3(6,4,1,1).$

In Section \ref{lvo}, we give Gorenstein terminal Fano varieties with conjecturally largest volume (Theorem \ref{largevolume}).

The existence of K$\mathrm{\ddot{a}}$hler-Einstein (KE) metrics on Fano varieties has attracted considerable interest. Some recent progress uses the ideas in higher dimension geometry, especially the methods from the minimal model program (MMP) (see survey \cite{XuSurvey2021}.
A lot of results are known about whether hypersurfaces in weighted projective space admit KE metrics \cite{JK01,JK2001,LiuPeracci}, \cite[Table 7]{KimOkadaWon}.
For del Pezzo surfaces which are quasi-smooth hypersurfaces in weighted
projective 3-spaces, 
Johnson and J.~Koll\'ar investigate the existence of a KE metric on many of these with Fano index one \cite{JK01}. Hwang and Yoon give some combinatorial ways to check whether a toric Fano variety or a fake weighted projective space has a KE metric \cite{HuwangYoon}. They show the projective space is the only weighted projective space that admits a KE metric \cite[Corollary 3.8]{HuwangYoon}. This implies the Fano varieties constructed in this paper do not admit KE metric. Fujita shows that for a $\mathbb{Q}-$Fano variety $X$ of dimension $n$ admitting KE metrics, the volume is bounded: 
$\vol(X)\leq \vol(\mathbb{P}^n)=(n+1)^n$ with equality iff $X\cong\mathbb{P}^n$ \cite[Theorem 1.1.]{Fuj15}. However, the volumes of the examples in this paper increase and exceed the bound $(n+1)^n$ very quickly as $n$ increases.

{\it Acknowledgements.} I would like to thank Burt Totaro for helpful conversations
and suggestions.

\section{Preliminaries}

For a collection of positive integers $a_0,\ldots,a_N$, we define the weighted projective space $X = \mathbb{P}^N(a_0,\ldots,a_N)$ to be the quotient variety
$(\mathbb{A}^{N+1}\setminus 0)/\mathbb{G}_m$, where the multiplicative group $\mathbb{G}_m$ acts by $t(x_0,\ldots,x_N) = (t^{a_0}x_0,\ldots,t^{a_N}x_N)$. the weighted projective space $X$ is called {\it well-formed} when the analogous
quotient stack $[(A^{n+1}-0)/\mathbb{G}_m]$
has trivial stabilizer group
in codimension 1. Equivalently, we say $X$ is {\it well-formed} if $\gcd(a_0,\ldots,\widehat{a_i},\ldots,a_n) = 1$ for each $i$
\cite[Definition 6.9]{Iano-Fletcher}.
From now on, we only consider well-formed weighted projective
spaces. 
For every integer $c$, we denote $\mathcal{O}(c)$ to be the sheaf
associated to a Weil divisor on $X$.
The sheaf $\mathcal{O}(c)$ is a line bundle if and only if every weight $a_i$ is a factor of $c$.
The canonical divisor of a well-form $X$
is given by $K_X=\mathcal{O}(-a_0-\cdots-a_N)$. A Weil divisor is defined to be ample if some positive multiple of it
is an ample Cartier divisor. The ample Weil divisor $\mathcal{O}(1)$ has volume $\frac{1}{a_0\cdots a_N}$.

We distinguish two notations. For integers $a, b$ and $r>0$, the notation $a \bmod r$ means an integer in $\{0,\cdots,r-1\}$. The notation $a \equiv b \ ( \bmod \ r)$ means $a$ and $b$ are same in the ring $\mathbb Z/r\mathbb Z$. For example, we have $4 \bmod 3$ is not equal to $7$, even though $4 \equiv 7 \ ( \bmod \ 3)$. 

The singularities of weighted projective spaces are all cyclic quotient singularities. Reid-Tai criterion \cite[Theorem 4.11]{Reidyoung} is a method to determine whether these are canonical or terminal. 

\begin{theorem} \label{RT}
For an integer $r>0$, let $\mu_r$ be the group of $r$th roots of unity with an action on affine space $\mathbb{A}^s$ by $\zeta(t_1,\ldots,t_s) = (\zeta^{b_1}t_1,\ldots,\zeta^{b_s}t_s)$.  We say quotient $\mathbb{A}^s/\mu_r$ to be a cyclic quotient singularity of type $\frac{1}{r}(b_1,\ldots,b_s)$.  Assume that 
$\gcd(r,b_1,\ldots,\widehat{b_i},\ldots,b_s) = 1$ for all $i = 1,\ldots,s$, which means the description is well-formed. 
Then the quotient singularity is canonical (resp.\ terminal) if and only if
$$\sum_{k = 1}^s tb_k \bmod r \geq r $$
(resp.\ $> r$) for all $t = 1,\ldots,r-1$.
\end{theorem}

\begin{remark}\label{open}
With an action of the torus $T=(\mathbb{G}_m)^{N+1}/\mathbb{G}_m\cong (\mathbb{G}_m)^N$
by scaling the variables, 
the weighted projective space $X=\mathbb{P}^N(a_0,\ldots,a_N)$ is a toric variety. The locus where $X$ is canonical (or terminal) is open and 
$T$-invariant. Thus if $X$ is canonical (or terminal) at a
point $q$, then $X$ is also canonical (or terminal) at all points $p$ such that $q$ is in the closure of the $T$-orbit of $p$. \cite[Lemma 2.2]{ETW21volume}. Therefore, in order to show $X$ is canonical (or terminal), it is enough to check that each coordinate point $[0:\cdots:0:1:0:\cdots:0]$ is canonical (or terminal).
\end{remark}

In this paper, we construct weighted projective spaces with large Fano indexes or large volumes using Sylvester's sequence. The sequence is given by $s_0=2$
and $s_n=s_{n-1}(s_{n-1}-1)+1$ for $n\geq 1$.
The first few numbers are $2$, $3$, $7$, $43$, $1807$.  We have $s_n=s_0\cdots s_{n-1}+1$, and hence the numbers
in the sequence are pairwise coprime. A crucial property is that the sum of the reciprocals
tends very quickly to 1, i.e.,
$$\frac{1}{s_0}+\frac{1}{s_1}+\cdots+\frac{1}{s_{n-1}}=1-\frac{1}{s_n-1}.$$
Also we have $s_n>2^{2^{n-1}}$
for all $n$, so the numbers grows doubly exponential with respespect to $n$.

\section{Large Fano index}\label{Lfanoindex}
In this section, we consider Fano varieties that are canonical, terminal or Gorenstein canonical and give examples with conjecturally largest Fano indexes.

Lemma 2.11 in \cite{ETW21volume} gives a trick of using certain subsets of weights to check that the singularity is canonical. With the similar argument, we give a trick to check that the singularity is terminal.
\begin{lemma}
\label{checkterminal}
Let $\frac{1}{r}(b_1,\ldots,b_s)$ be a well-formed quotient singularity. If there is some subset $I\subset\{1,\ldots,s\}$ such that $ \sum_{k\in I}b_k$ is a multiple of $r$, $\gcd(\{b_k| k\in I\}\cup \{r\}) = 1$ and $\gcd(b_i, r) = 1$ for some $i\in \{1,\ldots,s\} \setminus I$.  Then the singularity is terminal.
\end{lemma}

\begin{proof}
Since the singularity is well-formed, we may apply the Reid-Tai criterion.  Let $1 \leq t \leq r-1$ be an integer and consider
$$\sum_{k=1}^s tb_k  \bmod r = \sum_{k \in I} tb_k  \bmod r + \sum_{k \notin I}tb_k  \bmod r.$$
The first sum on the right-hand side is positive. Otherwise, the number $tb_k  \bmod r$ is zero for each $k\in I$, hence is a multiple of $r$. Since $1 \leq t \leq r-1$, this implies all $b_k$, $k\in I$ share a common factor with $r$, which contradicts to $\gcd(\{b_k| k\in I\}\cup \{r\}) = 1$. 
Also, the first sum on the right-hand side is a multiple of $r$ because $\sum_{k \in I} b_k$ is a multiple of $r$. Therefore, the first sum on the right-hand side is at least $r$.  
Moreover, since $\gcd(b_i, r) = 1$ and $1 \leq t \leq r-1$, the integer $tb_i$ is not a multiple of $r$. Thus $tb_i \bmod r$ is at least $1$. Therefore, the sum $\sum_{k=1}^s tb_k  \bmod r$ is larger than $r$.
\end{proof}

\begin{theorem}\label{largestindexcan}
For each integer $n\geq 2$, let $h=(s_{n-1}-1)(2s_{n-1}-3)$, $a_i=h/s_{n-i}$ for $2\leq i\leq n$, $a_1=s_{n-1}-1$ and $a_0=s_{n-1}-2$. Then the weighted projective space $X=\mathbb{P}^n(a_n,\ldots,a_0)=\mathbb{P}^n(h/s_0,\ldots,h/s_{n-2},s_{n-1}-1,s_{n-1}-2)$ is well-formed with canonical singularities and with Fano index $h$.
\end{theorem}
\begin{proof}
 We have $X$ is well-formed. Indeed,
 since $\gcd(s_{n-1}-1,s_{n-1}-2)=1$, we have $\gcd(a_n,\ldots,\widehat{a_i},\ldots,a_2,a_1,a_0)=1$ for $2\leq i\leq n$. 
 Since $\gcd(2(s_{n-1}-2),2s_{n-1}-3)=1$, we have $\gcd(s_{n-1}-2,2s_{n-1}-3)=1$.
 Together with $\gcd((s_{n-1}-1)/s_j,s_{n-1}-2)=1$, we get $\gcd(h/s_j,a_0)=1$ for $0\leq j\leq n-2$. Hence $ \gcd(a_n,\ldots,a_2,\widehat{a_1},a_0)=1.$
 Note that $\gcd(2(s_{n-1}-1),2(s_{n-1}-1)-1)=1$. So $\gcd(s_{n-1}-1,2s_{n-1}-3)=1$. 
 Since $\gcd((s_{n-1}-1)/s_0,\ldots,(s_{n-1}-1)/s_{n-2},s_{n-1}-1)=1$, we obtain $\gcd(a_n,\ldots,a_2,a_1,\widehat{a_0})=1.$

 The anti-canonical divisor $-K_X=\mathcal{O}(a_n+\cdots+a_0)=\mathcal{O}(h)$ since 
 \begin{equation*}
\begin{split}
&a_n+\cdots+a_0=h(\frac{1}{s_0}
+\cdots+\frac{1}{s_{n-2}})+2s_{n-1}-3 \\
&=(s_{n-1}-1)(2s_{n-1}-3)(1-\frac{1}{s_{n-1}-1})+2s_{n-1}-3\\
&=(s_{n-1}-1)(2s_{n-1}-3)=h.
\end{split}
\end{equation*}
Thus $X$ has Fano index $h$.
By Remark \ref{open}, in order to show $X$ has canonical singularities, it suffices to show the quotient singularity $\frac{1}{a_i}(a_n,\ldots,\widehat{a_i},\ldots,a_0)$ is canonical for $0\leq i \leq n$. 
Since $h$ is a multiple of $a_i$ for $1\leq i\leq n$, the sum $a_{n}+\cdots+\widehat{a_i}+\cdots+a_3+a_2+a_1+a_0=h-a_i$ is always a multiple of $a_i$ for $1\leq i \leq n$. Also since $\gcd(a_i,a_0)=1$ for $1\leq i \leq n$, we have $\frac{1}{a_i}(a_n,\ldots,\widehat{a_i},\ldots,a_0)$ is canonical for $1\leq i \leq n$ by Lemma 2.11 in \cite{ETW21volume}.
Note that 
\begin{equation*}
\begin{split}
a_n+\cdots+a_2&=h(\frac{1}{2}+\cdots+\frac{1}{s_{n-2}})=h(1-\frac{1}{s_{n-1}-1})\\
&=(s_{n-1}-1)(2s_{n-1}-3)-(2s_{n-1}-3)=(2s_{n-1}-3)a_0.
\end{split}
\end{equation*}
Also since $\gcd(a_0,a_n)=1$, we have $\frac{1}{a_0}(a_n,\ldots,a_1)$ is canonical by Lemma 2.11 in \cite{ETW21volume}.
 \end{proof}

When $n=2$, the weighted projective space $X=\mathbb{P}^2(3,2,1)$ has Fano index $6$
which is the Fano largest index among all weighted projective planes with canonical singularities by Brown and Kasprzyk \cite{database}. In Fact, we show $X=\mathbb{P}^2(3,2,1)$ has the largest Fano index in greater generality.

\begin{proposition}\label{cdelpezzo6}
Among all canonical del Pezzo surfaces, the weighted projective space $X=\mathbb{P}^2(3,2,1)$ has the largest Fano index $6$.   \end{proposition}

In order to show Proposition \ref{cdelpezzo6}, we need Lemma \ref{primitiveK} and Lemma
\ref{Picardrank1}.

\begin{lemma}\label{primitiveK}
Let $X$ be a smooth 
projective
surface and $Y$ be the blow-up of $X$ at a point. Then
$K_Y$ is always primitive, i.e., then there exists no element $A\in \Pic(Y)$ such that $K_Y\thicksim_{\mathbb{Q}}mA$ for some integer $m \geq 2$.    
\end{lemma}
\begin{proof}
Let $E$ be the exceptional divisor of the blow up. We have $K_Y\cdot E=-1$. If $K_Y\thicksim_{\mathbb{Q}}mA$  for some positive integer $m$ and $A\in \Pic(Y)$,
then $m(A\cdot E)=-1$. Since $Y$ is smooth, we have $A\cdot E$ is an integer. Hence $m=1$.
\end{proof}

\begin{lemma}\label{Picardrank1}
For a canonical del Pezzo surface $S$ with Picard number one, the Fano index $\mathrm{FI}(S)\leq6$.
\end{lemma}
\begin{proof}
Table I in \cite{MiyanishiZhang1988} gives the classification of canonical (equivalent to Gorenstein in dimension $2$) del Pezzo surfaces $S$ with Picard number one. We also know the canonical volume $(-K_S)^2$ from Table II in \cite{MiyanishiZhang1988}

Assume that $-K_S\thicksim_{\mathbb{Q}} mA$ for some integer $m>0$ and $A\in \mathrm{Cl}(S)$.
For all cases in Table I in \cite{MiyanishiZhang1988}, let $t=d$ if $\mathrm{Cl}(S)/\mathrm{Pic}(S)=\mathbb{Z}/d\mathbb{Z}$ for some integer $d>0$, $t=\mathrm{lcm}\{d_1, d_2\}$ if $\mathrm{Cl}(S)/\mathrm{Pic}(S)=\mathbb{Z}/d_1\mathbb{Z}\oplus \mathbb{Z}/d_2\mathbb{Z}$ for some integer $d_1>0$ and $d_2>0$ and $t=1$ if $\mathrm{Cl}(S)/\mathrm{Pic}(S)=0$. Then $(-K_S)^2=\frac{m^2}{t^2}(tA)^2$, where $(tA)^2\in \mathbb{Z}$ since $tA$ is Cartier. So $m^2=\frac{(-K_S)^2\cdot t^2}{(tA)^2}$. 
Let $u$ be the largest integer such that $u^2$ is a factor of $\frac{(-K_S)^2\cdot t^2}{(tA)^2}$. Then $m\leq u$. Therefore, the Fano index $\mathrm{FI}(S)\leq u$.
By Table I and II in \cite{MiyanishiZhang1988}, we have $u\leq 6$ except for the two cases that $S$ has Dynkin type of $2A_1+A_3$ and $S$ has Dynkin type of $D_5$.

By Table I and II in \cite{MiyanishiZhang1988}, if $S$ has Dynkin type of $2A_1+A_3$ or Dynkin type of $D_5$, then $(-K_S)^2=4$ and $4A\in \Pic(S)$. Hence $m^2=\frac{4\cdot 16}{(4A)^2}$, which implies $m\leq 6$ or $m=8$. 
It is sufficient to show that $m$ cannot be $8$.
Let $p:Y \rightarrow S$ be the minimal resolution of $S$. Then $K_Y=p^*K_S$ and $p^*(4A)\in\Pic(Y)$. 
If $-K_S\thicksim_{\mathbb{Q}}8A$, then $-K_Y\thicksim_{\mathbb{Q}}2p^*(4A)$. Note that $Y$ can be obtained by several blows up of points on smooth surfaces starting with $P^2$. So we get contradiction by Lemma \ref{primitiveK}.
\end{proof}

\begin{proof}[Proof of Proposition \ref{cdelpezzo6}]
Let $Z$ be a canonical del Pezzo surface.
By \cite[Lemma 2.]{MiyanishiZhang1988II}, there is a contraction $\pi:Z\rightarrow S$, where $S$ is a canonical del Pezzo surfaces with Picard rank one or two. We have $K_Z=\pi^*(K_S)+E$, where $E$ is a linear combination of exceptional divisors with integer coefficients. Then $\pi_*(K_Z)=K_S$. If $Z$ has Fano index larger than $6$, this means $K_Z\thicksim_{\mathbb{Q}}mA$ for some $A
\in \mathrm{Cl}(Z)$ and $m>6$. Hence $K_S\thicksim_{\mathbb{Q}}m\pi_*(A)$ with $\pi_*(A) \in \mathrm{Cl}(S)$, which implies Fano index $\mathrm{FI}(S)>6$. By Lemma \ref{Picardrank1}, we only need to show $\mathrm{FI}(S)\leq 6$ if $S$ has Picard rank two.

Now we assume $S$ has Picard rank two and $-K_S\thicksim_{\mathbb{Q}} mA$ for some integer $m>0$ and $A\in \mathrm{Cl}(S)$. 
By \cite[Lemma 2.]{MiyanishiZhang1988II}, we have $ 9-(K_Z)^2\geq (K_S)^2-(K_Z)^2$. So $9\geq (K_S)^2$.
If $S$ is $\mathbb{P}^1 \times \mathbb{P}^1$, 
then $\mathrm{Cl}(S)=\Pic(S)$. So $A$ is Cartier.
If $S$ is not $\mathbb{P}^1 \times \mathbb{P}^1$, all the possible Dynkin types that $S$ could have are given in \cite[Lemma 5.]{MiyanishiZhang1988II} as follows: $6A_1$,
$4A_1+A_3$, $4A_1$, $2A_1+D_4$, $2A_1+D_5$, $2A_3$, $A_3+D_4$, $D_4$, $D_6$, $D_7$.
Note the local class group of $A_n$, $D_n$($n$ even) and $D_n$($n$ odd) are $\mathbb{Z}/(n+1))\mathbb{Z}$, $\mathbb{Z}/2\mathbb{Z}\oplus \mathbb{Z}/2\mathbb{Z}$ and $\mathbb{Z}/4\mathbb{Z}$ respectively \cite[IV.24.]{Lipman1969}. Thus $4A$ is a cartier divisor. Hence we get $(4A)^2 \in \mathbb{Z}$. We have $\frac{m^2}{16}(4A)^2=(K_S)^2\leq 9$. So $m^2\leq \frac{9\cdot 16}{(4A)^2}$, which implies $m\leq 6$ or $m=12$.
It is sufficient to show that $m$ cannot be $12$.
Let $p:Y \rightarrow S$ be the minimal resolution of $S$. Then $K_Y=p^*K_S$ and $p^*(4A)\in\Pic(Y)$. 
If $-K_S\thicksim_{\mathbb{Q}}12A$, then $-K_Y\thicksim_{\mathbb{Q}}3p^*(4A)$. Note that $Y$ can be obtained by several blows up of points on smooth surfaces starting with $P^2$. So we get contradiction by Lemma \ref{primitiveK}.
So Fano index $\mathrm{FI}(S)\leq 6$.

The weighted projective space $\mathbb{P}(1,2,3)$ has anti-canonical class equals $\mathcal{O}(6)$, so has Fano index $6$.
\end{proof}

When $n=3$, the weighted projective space $X=\mathbb{P}^3(33,22,6,5)$ has Fano index $66$ which is known to be the largest Fano index among all weighted projective spaces of dimension $3$ with canonical singularities (see \cite[Table 3]{Kasprzyk2010} and \cite[2.6.]{ALNill}).
When $n=4$, the weighted projective space $X=\mathbb{P}^4(1743,1162,498,42,41)$ has Fano index $3486$ which is known to be the largest Fano index among all weighted projective spaces of dimension $4$ with canonical singularities \cite[Theorem 3.6 (ii)]{Kasprzyk2013}. 
 Therefore, we are motivated to conjecture:

\begin{conjecture}\label{conjectureFindexcanonical}
For each integer $n\geq 2$, let $X$ be the $n-$fold in Theorem \ref{largestindexcan}. Then the Fano index $\mathrm{FI}(X)=(s_{n-1}-1)(2s_{n-1}-3)$ of $X$ is the largest possible Fano index among all Fano $n-$folds with canonical singularities.
\end{conjecture}

 A general hypersurface of degree $h$ inside $X$ is conjectured to 
 have the minimal volume among all the canonical Calabi-Yau $n-$folds with ample Weil divisor $\mathcal{O}(1)$ \cite[Conjecture 1.2]{ETW21volume}. 
 Moreover, the hypersurface in $X$ given by a certain equation is Berglund-H\"ubsch-Krawitz (BHK) mirror to another hypersurface which gives the a klt Calabi-Yau pair with standard coefficients of conjecturally largest index $h$ \cite[Remark 3.7.]{ETW22index}.

\begin{theorem}\label{largestindex}
For each integer $n\geq 3$, let 
\begin{equation*}
\begin{split}
&a_0 = \frac{1}{2}(s_{n-1}-1)-1, \\
&a_1 = \frac{1}{2}(s_{n-1}-1),\\
&a_i = \frac{1}{2}(s_{n-1}-1)\frac{s_{n-1}-2}{s_{n-i}}\ \text{  for  } \ 2 \leq i\leq n-1, \\
&a_n=\frac{1}{2}\big(\frac{1}{2}(s_{n-1}-1)(s_{n-1}-2)-1)\big),
\end{split}
\end{equation*}
  Then the weighted projective space $X=\mathbb{P}^{n}(a_n,\ldots,a_0)$ is well-formed with terminal singularities and with Fano index $\frac{1}{2}(s_{n-1}-1)^2-1$. In particular, the Fano index is larger than $2^{2^{n-1}}$.
\end{theorem}
\begin{proof}
Since $s_{n-1}-1=s_0s_1\cdot \cdots \cdot s_{n-2}$, the number $\frac{s_{n-1}-1}{2s_{n-i}}$ is an integer for $2 \leq i\leq n-1$, which implies $a_i$ is an integer for $2 \leq i\leq n-1$. We can write $a_n=\frac{1}{2}\big(\frac{1}{2}(s_{n-1}-1)(s_{n-1}-1)-\frac{s_{n-1}-1}{2}-1\big)$.
We have $\frac{s_{n-1}-1}{2}=s_1\cdot \cdots \cdot s_{n-2}$ is odd since any two numbers in the Sylvester's sequence are relatively prime. Thus $\frac{1}{2}(-\frac{s_{n-1}-1}{2}-1)$ is an integer. So $a_n$ is an integer. It is clear that $a_0$ and $I$ are integers.

Since $a_0=a_1-1$, we have $\gcd(a_0,a_1)=1$.  Note that $\gcd(\frac{1}{2}(s_{n-1}-1)-1,\frac{1}{2}(s_{n-1}-1))=1$ implies $\gcd(\frac{1}{2}(s_{n-1}-1)-1,\frac{1}{2s_{n-i}}(s_{n-1}-1))=1$. We can write $a_0=\frac{(s_{n-1}-2)-1}{2}$. Not that $\gcd((s_{n-1}-2)-1, s_{n-1}-2)=1$. Hence $\gcd(\frac{(s_{n-1}-2)-1}{2},s_{n-1}-2)=1$. This implies $\gcd(a_0,a_i)=1$ for $2 \leq i\leq n-1$. We can write $a_n=\frac{1}{2}(s_{n-1}-2)(\frac{s_{n-1}-1}{2}+1)-a_1$. So $\gcd(a_1,a_n)=\gcd(a_1,\frac{1}{2}(s_{n-1}-2)(\frac{s_{n-1}-1}{2}+1))$. Note that $\gcd(s_{n-1}-1, s_{n-1}-2)=1$ implies $\gcd(\frac{s_{n-1}-1}{2},s_{n-1}-2)=1$.
Note that $\gcd(\frac{s_{n-1}-1}{2},\frac{s_{n-1}-1}{2}+1)=1$. Hence $\gcd(\frac{s_{n-1}-1}{2},(s_{n-1}-2)(\frac{s_{n-1}-1}{2}+1))=1$, which implies $
\gcd(a_1,a_n)=\gcd(a_1,\frac{1}{2}(s_{n-1}-2)(\frac{s_{n-1}-1}{2}+1))=1$. Thus $X$ is well-formed.

By Remark \ref{open}, in order to show $X$ has terminal singularities, it suffices to show the quotient singularity $\frac{1}{a_i}(a_n,\ldots,\widehat{a_i},\ldots,a_0)$ is terminal for $0\leq i \leq n$. 
Note that
\begin{equation*}
\begin{split}
&a_{n-1}+\ldots+a_2 =\frac{1}{2}(s_{n-1}-1)(s_{n-1}-2)(\frac{1}{s_1}+\ldots+\frac{1}{s_{n-2}})\\
&=\frac{1}{2}(s_{n-1}-1)(s_{n-1}-2)(\frac{1}{s_0}+\frac{1}{s_1}\ldots+\frac{1}{s_{n-2}})-\frac{1}{2s_0}(s_{n-1}-1)(s_{n-1}-2)\\
&=\frac{1}{2}(s_{n-1}-1)(s_{n-1}-2)(1-\frac{1}{s_{n-1}-1})-\frac{1}{2s_0}(s_{n-1}-1)(s_{n-1}-2)\\
&=\frac{1}{2s_0}(s_{n-1}-1)(s_{n-1}-2)-\frac{1}{2}(s_{n-1}-2).
\end{split}
\end{equation*}
By Lemma \ref{checkterminal}, we have $\frac{1}{a_n}(a_{n-1},\ldots,a_1,a_0)$ is terminal since $a_{n-1}+\cdots+a_2+a_0=a_n$ and $\gcd(a_1, a_n)=1$. Indeed, the sum 
\begin{equation*}
\begin{split}
&a_{n-1}+\cdots+a_2+a_0\\
&=\frac{1}{2s_0}(s_{n-1}-1)(s_{n-1}-2)-\frac{1}{2}(s_{n-1}-2)+\frac{1}{2}(s_{n-1}-1)-1\\
&=\frac{1}{2s_0}(s_{n-1}-1)(s_{n-1}-2)-\frac{1}{2}=a_n.
\end{split}
\end{equation*}
For $ 2 \leq i\leq n-1$, by Lemma \ref{checkterminal}, we have $\frac{1}{a_i}(a_n,a_{n-1},\ldots,\widehat{a_i},\ldots, a_2,a_1,a_0)$ is terminal since $a_n+a_{n-1}+\ldots+\widehat{a_i}+\ldots+a_2+a_1=(s_{n-i}-1)a_i$ and $\gcd(a_0, a_i)=1$. Indeed, the sum equals
\begin{equation*}
\begin{split}
&(a_{n-1}+\cdots+a_2)+a_n+a_1-a_i\\
&=\frac{1}{2s_0}(s_{n-1}-1)(s_{n-1}-2)-\frac{1}{2}(s_{n-1}-2)\\
&+\frac{1}{2}\big(\frac{1}{2}(s_{n-1}-1)(s_{n-1}-2)-1)\big)+\frac{1}{2}(s_{n-1}-1)-a_i\\
&=\frac{1}{2}(s_{n-1}-1)(s_{n-1}-2)-a_i\\
&=s_{n-i}\frac{1}{2s_{n-i}}(s_{n-1}-1)(s_{n-1}-2)-a_i=(s_{n-i}-1)a_i.
\end{split}
\end{equation*}
By Lemma \ref{checkterminal}, we have $\frac{1}{a_1}(a_n,a_{n-1},\ldots, a_2,a_0)$ is terminal since $a_n+a_{n-1}+\ldots+a_2+a_1=(s_{n-1}-3)a_1$ and $\gcd(a_0, a_1)=1$. Indeed, the sum $a_n+a_{n-1}+\ldots+a_2$ equals
\begin{equation*}
\begin{split}
&(a_{n-1}+\cdots+a_2)+a_n\\
&=\frac{1}{2s_0}(s_{n-1}-1)(s_{n-1}-2)-\frac{1}{2}(s_{n-1}-2)\\
&+\frac{1}{2}\big(\frac{1}{2}(s_{n-1}-1)(s_{n-1}-2)-1)\big)\\
&=\frac{1}{2}(s_{n-1}-1)(s_{n-1}-2)-\frac{1}{2}(s_{n-1}-1)=(s_{n-1}-3)a_1.
\end{split}
\end{equation*}
Since $a_0=\frac{1}{2}(s_{n-1}-1)-1$. The sum $(a_{n-1}+\cdots+a_2)+a_n$ equals
$$
\frac{1}{2}(s_{n-1}-1)(s_{n-1}-2)-\frac{1}{2}(s_{n-1}-1)=(\frac{1}{2}(s_{n-1}-1)-1)(s_{n-1}-1)=(s_{n-1}-1)a_0.
$$
Since $\gcd(a_0, a_1)=1$, then by Lemma \ref{checkterminal}, we have $\frac{1}{a_0}(a_n,a_{n-1},\ldots, a_2,a_1)$ is terminal.

Note that $-K_X=\mathcal{O}(a_n+\cdots+a_0)$,
where 
\begin{equation*}
\begin{split}
&a_n+\cdots+a_0=(a_{n-1}+\cdots+a_2)+a_n+a_1+a_0\\
&=\frac{1}{2}(s_{n-1}-1)(s_{n-1}-2)+a_0=\frac{1}{2}(s_{n-1}-1)^2-1.
\end{split}
\end{equation*}
Thus the Fano index of $X$ is $\frac{1}{2}(s_{n-1}-1)^2-1$.
\end{proof}
For $n=3$, the weight projective space $X=\mathbb{P}^3(7,5,3,2)$ has Fano index $17$, which is known to be the second largest Fano index for all $\mathbb{Q}-$Fano threefolds  \cite[Theorem 1.4.]{ProkhorovQ-FanoI} \cite[Theorem 1.2.]{ProkhorovQ-FanoII}. 
For $n=4$, the weight projective space 
$X=\mathbb{P}^4(430,287,123,21,20)$ has Fano index $881$, which is known to be the largest Fano index among all well-formed weighted projective spaces with terminal singularities in dimension $4$ \cite[Theorem 3.5.]{Kasprzyk2013} \cite[Section 3.4.]{Brown-Kasprzyk}. Therefore, we are motivated to conjecture:

\begin{conjecture}\label{conjectureFindexterminal}
For each integer $n\geq 4$, let $X$ be the $n-$fold in Theorem \ref{largestindex}. Then the Fano index $\mathrm{FI}(X)=\frac{1}{2}(s_{n-1}-1)^2-1$ of $X$ is the largest possible Fano index among all Fano $n-$folds with terminal singularities.
\end{conjecture}

Nill found the weighted projective space $X$ in Theorem \ref{Gocanonicalindex} \cite[Corollary 6.1.]{Nill2007}. We are just proving its properties for the reader's convenience.

\begin{theorem}\label{Gocanonicalindex}
For each integer $n\geq 1$, let $h=s_n-1$, $a_i=h/s_{n-i}$ for $1\leq i\leq n$ and $a_0=1$. Then the weighted projective space $X=\mathbb{P}^n(a_n,\ldots,a_0)=\mathbb{P}^n(h/s_0,\ldots,h/s_{n-1},1)$ is well-formed with Gorenstein canonical singularities and with Fano index $h$.
\end{theorem}

\begin{proof}
We have $X$ is well-formed since $\gcd(a_n,\ldots,\widehat{a_i},\ldots,a_2,a_1,1)=1$ for $1\leq i\leq n$ and $\gcd(h/s_0,\ldots,h/s_{n-1})=1$.
 The anti-canonical divisor $-K_X=\mathcal{O}(a_n+\cdots+a_0)=\mathcal{O}(h)$ since 
 $$a_n+\cdots+a_0=h(\frac{1}{s_0}
+\cdots+\frac{1}{s_{n-1}})+1=(s_n-1)(1-\frac{1}{s_n-1})+1=h.$$
Thus the Fano index of $X$ is $h$. Note that each weightes divides $h$, so $K_X$ is a Cartier divisor, i.e., $X$ is Gorenstein.
Also $X$ is canonical by \cite[Corollary 2.10]{ETW21volume}.
\end{proof}

Corollary 6.1 in \cite{Nill2007} shows that $X$ given by Theorem \ref{Gocanonicalindex} has largest Fano index among all well-formed weighted projective spaces of dimension $n$ with Gorenstein canonical singularities. Therefore, we have the following conjecture:

\begin{conjecture}\label{conjectureFindexGorcanonical}
For each integer $n\geq 1$, let $X$ be the $n-$fold in Theorem \ref{Gocanonicalindex}. Then the Fano index $\mathrm{FI}(X)=s_n-1$ of $X$ is the largest possible Fano index among all Fano $n-$folds with Gorenstein canonical singularities.
\end{conjecture}

\section{Large Volume}\label{lvo}
In this section, we will find Gorenstein terminal Fano varieties with large volume. 
When dimension $n$ is lower and even, we know the optimal example among all Gorenstein terminal weighted projective spaces of dimension $n$ with large volume. We list them as follows. 
For $n=4,6,8,10$, the weight projective spaces 
\begin{equation*}
\begin{split}
&\mathbb{P}^4(2,1,1,1,1) \text{ with volume } 648,\\
&\mathbb{P}^6(8,6,4,3,1,1,1) \text{ with volume } 331776,\\
&\mathbb{P}^8(140,105,84,60,15,10,4,1,1) \text{ with volume } 21781872000,\\
&\mathbb{P}^{10}(16328, 12246,8164,6123,3768,1884,312,156,1,1,1)
\text{ with volume } \\
&23029100604532998144
\end{split}
\end{equation*}
have the largest volume among all Gorenstein terminal weighted projective spaces in dimension $n$ \cite[Table 5.]{Kasprzyk2013}.
When $n$ is even, we do not find a way to generalize these known examples to get varieties in higher dimensions.
However, when the dimension $n$ is odd, we have generalization in higher dimensions.

\begin{theorem}\label{largevolume}
For each odd integer $n=2k+1\geq 5$, where integer $k\geq 2$, let 
\begin{equation*}
\begin{split}
&h=2s_0s_1 \cdots s_{k-1}=2(s_k-1),\\
&a_0 = a_1=a_2=1, \\
&a_{2i-1} = \frac{h}{2s_{k+1-i}}=s_0s_1 \cdots \widehat{s_{k+1-i}} \cdots s_{k-1} \ \text{  for  } \ 2 \leq i\leq k-1 \text{  when  } k
\geq 3,\\
&a_{2i} = \frac{h}{s_{k+1-i}}=2s_0s_1 \cdots \widehat{s_{k+1-i}} \cdots s_{k-1} \ \text{  for  } \ 2 \leq i\leq k-1 \text{  when  } k
\geq 3, \\
&a_{n-2}=h/6=s_0s_2\cdots s_{k-1},\\
&a_{n-1}=h/4=s_1s_2\cdots s_{k-1},\\
&a_n=h/3=2s_0s_2\cdots s_{k-1},
\end{split}
\end{equation*}  
Then Gorenstein terminal weighted projective space $X=\mathbb{P}^{n}(a_n,\ldots,a_0)$ has volume $(-K_X)^n=2^{\frac{n+1}{2}}(s_{\frac{n-1}{2}}-1)^4$.
\end{theorem}

\begin{proof}
That $a_0 = a_1=a_2=1$ implies $X$ is well-formed. The anti-canonical divisor $-K_X=\mathcal{O}(a_n+\cdots+a_0)=\mathcal{O}(h)$ since $a_n+\cdots+a_0$ equals
\begin{equation*}
\begin{split}
&h(\frac{1}{3}+\frac{1}{4}+\frac{1}{6}+\frac{1}{7}+\frac{1}{14}+\frac{1}{43}+\frac{1}{86}+\cdots+\frac{1}{s_{k-1}}+\frac{1}{2s_{k-1}})+3\\
&=h(\frac{1}{2}+\frac{1}{3}+\frac{1}{7}+\frac{1}{43}+\cdots+\frac{1}{s_{k-1}}-\frac{1}{2})+\frac{h}{2}(\frac{1}{2}+\frac{1}{3}+\frac{1}{7}+\cdots+\frac{1}{s_{k-1}})+3\\
&=h(1-\frac{1}{s_k-1}-\frac{1}{2})+\frac{h}{2}(1-\frac{1}{s_k-1})=h(1-\frac{3}{2(s_k-1)})+3=2(s_k-1).
\end{split}
\end{equation*}
Thus $X$ has Fano index $h=2(s_k-1)$ and is Gorenstein since $h$ is a multiple of every weight. 

By Remark\ref{open}, in order to show $X$ has terminal singularities, it suffices to show the quotient singularity $\frac{1}{a_j}(a_n,\ldots,\widehat{a_j},\ldots,a_0)$ is terminal for $3\leq j \leq n$.
If $j=2i-1$ for $ 2 \leq i\leq k-1$, we have $a_{j+1}=2a_j$. 
If $j=2i$ for $ 2 \leq i\leq k-1$, we have 
\begin{equation*}
\begin{split}
a_{2i-1}+a_{n-1}+a_{n}&=\frac{h}{2s_{k+1-i}}+\frac{h}{2\cdot 3}+\frac{h}{3}=h\frac{3+s_{k+1-i}+2s_{k+1-i}}{2\cdot 3 \cdot s_{k+1-i}}\\
&=\frac{1+s_{k+1-i}}{2}\frac{h}{s_{k+1-i}}=\frac{1+s_{k+1-i}}{2}a_{2i}.
\end{split}
\end{equation*}
Since $\gcd(a_0, a_j)=1$, we have $\frac{1}{a_j}(a_n,\ldots,\widehat{a_j},\ldots,a_0)$ is terminal by Lemma \ref{checkterminal} for $j=2i-1$ and $j=2i$, where $ 2 \leq i\leq k-1$. 
Similarly, since $a_n=2a_{n-2}$ and $\gcd(a_0, a_{n-2})=1$, we have $\frac{1}{a_{n-2}}(a_n,a_{n-1},\widehat{a_{n-2}},a_{n-3},\ldots,a_0)$ is terminal.

We have $\frac{1}{a_{n-1}}(a_n,\widehat{a_{n-1}},a_{n-2},\ldots,a_0)$ is terminal by Lemma \ref{checkterminal}.
Indeed, we have $a_{n-2}+a_{n}=\frac{h}{6}+\frac{h}{3}=\frac{h}{2}=2a_{n-1}$ and $\gcd(a_0, a_{n-1})=1$.

We also have $\frac{1}{a_n}(\widehat{a_n},a_{n-1},a_{n-2},\ldots,a_0)$ is terminal by Lemma \ref{checkterminal}.
Indeed, we have 
\begin{equation*}
\begin{split}
\sum_{i=2}^{k-1}a_{2i-1}+a_{n-1}+a_1&=\frac{h}{2}(\frac{1}{s_{k-1}+\cdots+\frac{1}{3}}+\frac{1}{2})-\frac{h}{2\cdot3}+1\\
&=\frac{h}{2}(1-\frac{1}{s_k-1})-\frac{h}{2\cdot3}+1=\frac{h}{3}=a_n
\end{split}
\end{equation*}
and $\gcd(a_0, a_n)=1$.

Now we compute the volume $(-K_X)^n$, which equals
\begin{equation*}
\begin{split}
&\frac{h^n}{\frac{h}{3}\cdot\frac{h}{4}\cdot\frac{h}{6}\cdot\frac{h}{7}\cdot\frac{h}{14}\cdot\frac{h}{43}\cdot\frac{h}{86}\cdots\frac{h}{s_{k-1}}\cdot\frac{h}{2s_{k-1}}}\\
&=\frac{h^n}{\frac{2h^k}{2\cdot 3\cdot 7\cdot43\cdot s_{k-1}}\cdot(\frac{h}{2})^k\cdot\frac{1}{2\cdot 3\cdot 7\cdot43\cdot s_{k-1}}}\\
&=h^2\cdot 2^{k-1}(s_k-1)^2=2^{k+1}(s_k-1)^4=2^{\frac{n+1}{2}}(s_{\frac{n-1}{2}}-1)^4.
\end{split}
\end{equation*}
\end{proof}
For $n=5,7,9$, Theorem \ref{largevolume} gives weight projective spaces 
\begin{equation*}
\begin{split}
&\mathbb{P}^5(4,3,2,1,1,1) \text{ with volume } 10368,\\
&\mathbb{P}^7(28,21,14,12,6,1,1,1) \text{ with volume } 49787136,\\
&\mathbb{P}^9(1204,903,602,516,258,84,42,1,1,1) \text{ with volume } 340424620687872
\end{split}
\end{equation*}
respectively. They have largest volume among all Gorenstein terminal weighted projective spaces in dimension $n$ \cite[Table 5.]{Kasprzyk2013}. The results motivate us to conjeture:

\begin{conjecture}
For each odd integer $n\geq 5$, let $X$ be the $n-$fold in Theorem \ref{largevolume}. Then the volume $(-K_X)^n=2^{\frac{n+1}{2}}(s_{\frac{n-1}{2}}-1)^4$ of $X$ is the largest possible volume among all Gorenstein terminal Fano $n-$folds.
\end{conjecture}

\end{document}